\documentclass[11pt]{amsproc}

\usepackage{amsmath, amssymb, amsthm, latexsym,nicefrac,setspace}
\usepackage[pagebackref,colorlinks,linkcolor=blue,citecolor=blue,urlcolor=blue,hypertexnames=true]{hyperref}
\usepackage{amsrefs}
\usepackage{thmtools}

\usepackage[usenames]{color}
\newcommand{\MN}[1]{{\color{red}{\noindent #1}}}

\newcommand{\grpone}{1}

\theoremstyle{plain}
\newtheorem{theorem}{Theorem}[section]
\newtheorem{question}[theorem]{Question}
\newtheorem{lemma}[theorem]{Lemma}

\declaretheorem[title=Conjecture,numberlike=theorem]{conjecture}

\theoremstyle{remark}
\newtheorem{remark}[theorem]{Remark}
\newtheorem{example}[theorem]{Example}

\input{xy}
\xyoption{all}

\def\Z{\mathbb Z}

\author{Martin Nitsche}
\address{Martin Nitsche, TU Dresden, Germany}
\email{martin.nitsche@tu-dresden.de}
\author{Andreas Thom}
\address{Andreas Thom, TU Dresden, Germany}
\email{andreas.thom@tu-dresden.de}

\title{Universal solvability of group equations}

\begin{document}

\onehalfspace

\begin{abstract}
We give a new criterion for solvability of group equations, providing proofs of various generalizations of the Kervaire--Laudenbach Conjecture for Connes-embeddable groups. 
\end{abstract}

\maketitle

\section{Introduction}
Let $G$ be a group and let $w_1, w_2, \dots,w_k \in \mathbb F_n \ast G$, the free product of $G$ with the free group on $n$ letters.
We want to study the question under what conditions $w_1,\dots,w_k$, considered as equations with $n$ variables and constants from $G$, can be solved simultaneously in a group containing $G$, or, equivalently, when the natural homomorphism $G\to (\mathbb{F}_n\ast G)/\langle\!\langle w_1,\dots,w_k\rangle\!\rangle$ is injective.
In this article we are interested in conditions that depend only on $\varepsilon(w_1),\dots,\varepsilon(w_k)$, the images under the augmentation homomorphism $\varepsilon \colon \mathbb F_n \ast G \to \mathbb F_n$ that sends $G$ to the neutral element. Following \cite{klyachkothom2017new}, we call this data the {\it content} of the system of equations. If such a condition is met, we say that the \emph{universal} system of group equations with content $\varepsilon(w_1),\dots,\varepsilon(w_k)$ is solvable over $G$.
Questions about the solvability of group equations have a long history going back to \cite{neumann}, see also \cite{klyachkothom2017new} for more background and references. In the case of one variable and one equation $w \in \mathbb Z \ast G$, the famous Kervaire--Laudenbach Conjecture predicts that $w$ can be solved in a group containing $G$ if $\varepsilon(w) \neq 0 \in \mathbb Z$. 
This has been generalized by Klyachko and the second author as follows:

\begin{conjecture}[Generalized Kervaire--Laudenbach Conjecture, \cite{klyachkothom2017new}]\label{conj:generalized-KLC}
Let $G$ be any group and $w \in \mathbb F_n \ast G$.
If $\varepsilon(w)\neq \grpone\in \mathbb{F}_n$, i.e.\ the content of the equation is non-trivial, then the single equation $w(x_1,\ldots,x_n)=\grpone$ with $n$ variables and constants from $G$ can be solved in a group $H$ containing $G$. Moreover, if $G$ is finite, then $H$  can be taken to be finite.
\end{conjecture}

The crux of the matter is that when forming $\varepsilon(w)$ by deleting the constants from $G$, there may be a lot of cancellation among the variables and their inverses, so that the structure of $w$ can be considerably more complicated. 
The combinatorial approach pursued, among others, by Howie and Gersten \cites{MR919828, MR614523}, leads to positive results only under constraints on the {\it unreduced} words obtained by deleting the constants.
However, note that Klyachko's combinatorial methods in \cites{MR1218513,MR2251364} do allow for cancellation, but they assume that $G$ is torsion-free.
 An entirely different method that naturally sidesteps the complications of cancellation is the topological approach by Gerstenhaber and Rothaus \cite{MR0166296}, which assumes that $G$ embeds into a unitary group $\mathrm{U}(n)$. In \cite{MR2460675}, Pestov made the connection with Connes' Embedding Problem and brought their result into the following form:

\begin{theorem}[Gerstenhaber--Rothaus \cite{MR0166296}, Pestov \cite{MR2460675}]\label{thm:gerstenhaber-rothaus}
Let $G$ be a Connes-embeddable group and let $w_1,\dots,w_k \in \mathbb F_n \ast G$. If the presentation complex of the presentation
\[\langle x_1,\dots,x_n \mid \varepsilon(w_1),\dots,\varepsilon(w_k) \rangle\] has trivial second homology, then the system of equations $w_1,\dots,w_k$ is solvable in a group $H$ containing $G$. If $G$ is finite, then $H$ can be taken to be finite.
\end{theorem}

Here, a Connes-embeddable group is by definition a group which can be embedded into a certain metric ultraproduct of unitary groups, see \cites{MR2460675, icm} for more background on this topic.
This proves the original Kervaire--Laudenbach Conjecture for Connes-embeddable groups, a large class of groups which includes all sofic groups, and to which to date no counterexamples are known. In regard to the \autoref{conj:generalized-KLC}, however, \autoref{thm:gerstenhaber-rothaus} only applies to a limited class of equations.

We will combine the topological approach with combinatorial methods and a compactness argument in order to show the following theorem, which is our main result. The wreath product construction, which plays a central role in our proof, was first applied by Levin \cite{MR142643} to specific equations with only a single variable.

\begin{theorem}\MN{\label{thm:main-result}}
Let $G$ be a Connes-embeddable group and let $w_1,\dots,w_k \in \mathbb F_n \ast G$. If the presentation complex of the presentation
$$ \langle x_1,\dots,x_n \mid \varepsilon(w_1),\dots,\varepsilon(w_k) \rangle$$ admits a covering with trivial second homology, then the system of equations $w_1,\dots,w_k$ is solvable in a group containing $G$.
\end{theorem}

This result applies to many more systems of equations than \autoref{thm:gerstenhaber-rothaus}. In particular, it proves the first part of \autoref{conj:generalized-KLC} for Connes-embeddable groups.
This was previously only known in the case that $\varepsilon(w) \not \in [\mathbb F_n,[\mathbb F_n,\mathbb F_n]]$, see the main result of \cite{klyachkothom2017new}.
If Howie's conjecture in \cite{MR614523} is true and \autoref{thm:gerstenhaber-rothaus} holds for all groups $G$, then the condition that $G$ is Connes-embeddable can be dropped in \autoref{thm:main-result}.
As of now, it seems unclear to us if the the condition on existence of a covering with trivial second homology is the optimal condition. However, \autoref{gersten} shows that already the slightly weaker assumption of vanishing of the second Hurewicz map is not enough to imply the conclusion.

\smallskip

The proof for the case of one equation, which contains the essential new idea, was found by the first author with a slightly different argument, making essential use of orderability of one-relator groups. The present joint paper contains a reduction to a purely homological condition and arose when exploiting the newly found method in order to cover the case of many equations.

\section{Proof of the main result}

First of all, we observe that the results of Gerstenhaber--Rothaus and Pestov extend to systems of equations with infinitely many variables and equations: Let $G$ be a group and $X$ a set of variables, possibly infinite. Let $\mathbb{F}_X$ be the free group over $X$, and $W\subset\mathbb{F}_X\ast G$ a set of equations, possibly infinite. Then we call $(X,W)$ a system of equations over $G$.

\begin{lemma}\label{lem:infinite-GR}
Let $(X,W)$ be a system of equations over a group $G$ and assume that the presentation complex of the presentation $\langle X\mid\varepsilon(W)\rangle$ has trivial second homology.
\begin{enumerate}
\item\label{itm:finite}
If $G=\mathrm{U}(m)$ for some $m \in \mathbb N$, then the system of equations can be solved in $G$.
\item\label{itm:connes-embeddable}
If $G$ is Connes-embeddable, then the system of equations can be solved in some Connes-embeddable group $H$ containing $G$.
\end{enumerate}
\end{lemma}
\begin{proof}
(1) If $F\subset W$ is a finite subset and $X_F\subset X$ is the finite subset of variables that appear in $F$, then the presentation complex of $\langle X_F\mid \varepsilon(F)\rangle$ is a subcomplex of the presentation complex of $\langle X\mid \varepsilon(W)\rangle$ and hence has trivial second homology. By the result of Gerstenhaber--Rothaus, $(X_F,F)$ has a solution $z_F\colon X_F\to\mathrm{U}(m)$. In this way we obtain, for every $x\in X$, a net that assigns to each finite subset $F\subset W$ a solution $z_F(x)$. Because $\mathrm{U}(m)$ is compact, we can pick for every net an accumulation point $z(x)$. The assignment $x\mapsto z(x)$ solves the system of equations.

(2) By definition, the group $G$ is Connes-embeddable iff it embeds into some ultraproduct $\left(\prod_{i\in I} {\rm U}(m_i)\right)/\sim$, where the equivalence relation is convergence with regard to the normalized Hilbert--Schmidt norm and some ultrafilter on $I$.
We can lift the constants in the system of equations $(X,W)$ to $\prod_{I} {\rm U}(m_i)$ and then apply part (1) to all ${\rm U}(m_i)$--factors simultaneously to obtain a solution for the new system of equations in $\prod_{I} {\rm U}(m_i)$. The image of this solution in $\left(\prod_{i\in I} {\rm U}(m_i)\right)/\sim$ is a solution to the original system of equations.
\end{proof}

\begin{remark} \label{newrem} 
Part (\ref{itm:connes-embeddable}) of \autoref{lem:infinite-GR} also holds if
the coefficient group $G$ is not Connes-embeddable but still satisfies the conclusion of the theorem of Gerstenhaber--Rothaus, namely that all finite systems of equations are solvable over $G$ if the second homology of the associated presentation complex is trivial. To pass to infinite systems of equations, compactness can be replaced by an ultrafilter argument.
However, compared to \autoref{lem:infinite-GR}, we get even less control over the group $H$ in this case.
\end{remark}

One can associate to any system of equations $(X,W)$ a cell complex $Y$ with a single $0$-cell $c_0$, one (oriented) $1$-cell $c_x$ for every variable $x\in X$, and one $2$-cell $c_w$ for every equation $w\in W$, glued to the $1$-skeleton according to the appearances of the letters from $X$ in $w$, where we ignore the constants from $G$.
This cell complex depends on the specific system of equations, but up to homotopy it is the same as the presentation $2$-complex of $\langle X\mid\varepsilon(W)\rangle$.

For any subgroup $\pi$ of $\Gamma=\pi_1(Y)=\langle X\mid\varepsilon(W)\rangle$ there is a corresponding covering complex $\bar{Y}\to Y$. We mimic this construction to build the ``$\pi$-covering'' $(\bar{X},\bar{W})$ of the system of equations $(X,W)$: 
Its set of variables is $\bar{X}=\pi\backslash\Gamma\times X$. To obtain the set of equations $\bar{W}$ we use the canonical right action $\pi\backslash\Gamma\curvearrowleft\Gamma\leftarrow\mathbb{F}_X\leftarrow\mathbb{F}_X\ast G$ to construct, inductively, the set-theoretic map $\psi\colon\pi\backslash\Gamma\times(\mathbb{F}_X\ast G)\to\mathbb{F}_{\bar{X}}\ast G$ by letting for all $[\gamma]\in\pi\backslash\Gamma$, $g\in G$, $x\in X$, $l_1,l_2\in\mathbb{F}_X\ast G$
\begin{equation}\label{eqn:covering-equations}
\begin{gathered}
\psi([\gamma],g)=g,\quad \psi([\gamma],x)=([\gamma],x),\quad \psi([\gamma],x^{-1})=([\gamma]. x^{-1},x)^{-1},\\
\psi([\gamma],l_1\cdot l_2)=\psi([\gamma],l_1)\cdot\psi([\gamma].l_1,l_2).
\end{gathered}
\end{equation}
Then we set $\bar{W}=\{\psi([\gamma],w)\mid [\gamma]\in\pi\backslash\Gamma,w\in W\}$. Note that $\psi$ is naturally a cocycle, which is determined by its tautologically chosen values on $X \cup G$.

Using the canonical correspondence between the $0$-cells of $\bar{Y}$ and the right cosets $\pi\backslash\Gamma$, we can match the variables $\bar{X}$ to the $1$-cells of $\bar{Y}$, assigning $([\gamma],x)\in\bar{X}$ to that lift of $c_x$ which starts at the $[\gamma]$--th $0$-cell and ends at the $[\gamma].x$--th. Then, any equation $\psi([\gamma],w)\in\bar{W}$ describes, after removing the constants, the attaching map of a $2$-cell in $\bar{Y}$, namely that lift of $c_w$ for which the attaching map starts at the $[\gamma]$--th $0$-cell. Consequently, the 2\nobreakdash-complex associated to $(\bar{X},\bar{W})$ can be obtained from $\bar{Y}$ by collapsing the $0$-skeleton to a single vertex.

Now, we have the following crucial observation:
\begin{lemma}\label{lem:covering}
If $(\bar{X},\bar{W})$ is solvable in a group $H$ containing $G$, then $(X,W)$ is solvable in the group $\bar{H}:=\big(\prod_{\pi\backslash\Gamma}H\big)\rtimes\Gamma$.
\end{lemma}
\begin{proof}
First, recall that the multiplication and group inversion in $\bar{H}$ are given by
\begin{equation}\label{eqn:semidirect-product}
\begin{gathered}
\Big([\gamma]\mapsto h_{[\gamma]}, \lambda\Big)\cdot\Big([\gamma]\mapsto h'_{[\gamma]}, \lambda'\Big)=\Big([\gamma]\mapsto h_{[\gamma]}h'_{[\gamma].\lambda}, \lambda\lambda'\Big),\\
\Big([\gamma]\mapsto h_{[\gamma]}, \lambda\Big)^{-1}=\Big([\gamma]\mapsto (h_{[\gamma].\lambda^{-1}})^{-1}, \lambda^{-1}\Big).
\end{gathered}
\end{equation}
Assume that a solution to $(\bar{X},\bar{W})$ is given in the form of a homomorphism $\bar{z}\colon\mathbb{F}_{\bar{X}}\ast G\to H$ that is injective on $G$ and trivial on each $\bar{w}\in\bar{W}$. Then we check that the homomorphism $z\colon\mathbb{F}_X\ast G\to\bar{H}$ given by
\begin{equation*}
g\mapsto\Big([\gamma]\mapsto\bar{z}(g),\grpone\Big),\qquad
x\mapsto\Big([\gamma]\mapsto\bar{z}(([\gamma],x)),\lambda_x\Big)\qquad\forall g\in G, x\in X
\end{equation*}
is a solution to $(X,W)$, where $\lambda_x\in\Gamma$ denotes the generator corresponding to $x$.
Indeed, $z$ it is injective on $G$, and it maps each $w\in W$ into the subgroup $\prod H<\bar{H}$ because $\varepsilon(w)$ is a relation in $\Gamma$.
Let $\phi\colon\pi\backslash\Gamma\times(\mathbb{F}_X\ast G)\to H$ be the set-theoretic map that sends $([\gamma],l)$ to the $[\gamma]$--th coordinate of the $\big(\prod H\big)$--part of $z(l)\in\bar{H}$. From \autoref{eqn:semidirect-product} it follows that for all $[\gamma]\in\pi\backslash\Gamma$, $g\in G$, $x\in X$, $l_1,l_2\in\mathbb{F}_X\ast G$
\begin{equation*}
\begin{gathered}
\phi([\gamma],g)=\bar{z}(g),\quad \phi([\gamma],x)=\bar{z}(([\gamma],x)),\quad \phi([\gamma],x^{-1})=\bar{z}(([\gamma]. x^{-1},x)^{-1}),\\
\phi([\gamma],l_1\cdot l_2)=\phi([\gamma],l_1)\cdot\phi([\gamma].l_1,l_2).
\end{gathered}
\end{equation*}
Comparing this to \autoref{eqn:covering-equations}, it is clear that $\phi=\bar{z}\circ\psi$ and hence $\phi([\gamma],w)=\bar{z}(\bar{w})=\grpone$ for all $[\gamma]\in\pi\backslash\Gamma$, $w\in W$. Therefore, $z$ is trivial on all $w\in W$.
\end{proof}

We are now ready to prove our main result.

\begin{proof}[Proof of \autoref{thm:main-result}:]
Let $(X,W)=\big(\{x_1,\ldots,x_n\},\{w_1,\ldots,w_k\}\big)$ be the given system of equations, $Y$ its associated cell complex, and $\Gamma=\langle X\mid \varepsilon(W)\rangle$ the associated group presentation.
Let $\bar{P}\to P$ be the covering of the presentation complex from the assumption, $\pi<\Gamma$ be the associated inclusion of fundamental groups, and $(\bar{X},\bar{W})$ the covering system of equations obtained from this group inclusion.

Then the cell complex associated to $(\bar{X},\bar{W})$ has the same second homology as the covering $\bar{Y}\to Y$ associated to $\pi<\Gamma$. Because $\bar{Y}$ is homotopy equivalent to $\bar{P}$, this homology group is trivial. Hence, by \autoref{lem:infinite-GR}, the system $(\bar{X},\bar{W})$ has a solution over $G$. Now, using \autoref{lem:covering}, so does the original system $(X,W)$.
\end{proof}

\begin{remark}
The condition that the presentation complex has a covering with trivial second homology can be reformulated in terms of homology with local coefficients:
It is met iff there exists a $\Gamma$-set $Z$ such that the presentation complex has trivial second homology with local coefficients in $\Z[Z]$. A priori, this is an interesting ring theoretic condition on the second differential of the equivariant chain complex.
\end{remark}

\begin{remark}
Note that, unlike in \autoref{lem:infinite-GR}, we cannot conclude that the group $\bar{H}$ in which the system of equations from \autoref{thm:main-result} can be solved is Connes-embeddable. However, it seems likely that $\bar H$ can be chosen to satisfy the conclusion of \autoref{thm:gerstenhaber-rothaus}. This will be subject of further investigation.
\end{remark}

\section{Applications and open questions}\label{sec:applications}
There are a few special cases, when the condition of \autoref{thm:main-result} is always satisfied:
\begin{enumerate}
\item When the presentation complex itself has trivial second homology, the theorem reduces to the classical result of Gerstenhaber--Rothaus. This can only occur if there are at most as many equations as variables.

\item When the presentation complex is aspherical. One large class of examples where this happens are the presentation complexes of torsion-free small cancellation groups, see for example \cite{huebschmann}.
In this case the number of equations can be larger than the number of variables, or even infinite (the theorem adapts to the infinite case without modifications), which is counter intuitive at first sight. This should lead to some interesting examples.

\item\label{item:one-relator-groups} When there is a single (non-trivial) equation $w$, i.e.\ when $\Gamma$ is a one-relator group. If $\Gamma$ is torsion-free, the presentation complex is aspherical. If $\Gamma$ has torsion, it follows $\varepsilon(w)=z^r$ for some $z\in\mathbb{F}_n$ not a proper power, and we let $\Gamma'=\langle x_1,\dots,x_n\mid z\rangle$ be the torsion-free one-relator quotient. Then, the covering of the presentation complex corresponding to the subgroup $\pi=\mathrm{ker}(\Gamma\to\Gamma')$ has trivial second homology. Indeed, the second differential in the cellular chain complex of this covering differs only by a factor of $r$ from the differential in the chain complex of the universal covering of the presentation complex of $\Gamma'$.

\item Another example where the presentation complex is aspherical is when $k=n-1$ and the first $\ell^2$--Betti number of the group $\Gamma$ vanishes. In this case the second homology of the universal covering of the presentation complex embeds into the second $\ell^2$-homology, which vanishes due to vanishing of the Euler characteristic, see \cite{berhil}.
\end{enumerate}

When we want to show that a given universal system of group equations is solvable over Connes-embeddable groups,
\autoref{thm:main-result} is currently the best known result.
In the case of a single relation $\varepsilon(w)$ the criterion is also sharp as one can easily see. For more than one relation we do not know, but we suspect that there are universal systems of group equations that do not meet the conditions of the theorem but are still solvable.
One simple example where this might be the case is the universal system of equations with variables $a,b,c$ and content $[a,b], [b,c], [c,a]$. Here, the second homotopy group of the presentation complex is isomorphic to $\Z[\Z^3]$, generated by the Hall--Witt identity. This identity gives rise to a non-vanishing homology class on every non-trivial covering.

\begin{question} \label{q1}
Can the universal system of group equations in three variables $a,b,c$ with content $[a,b], [b,c], [c,a]$ be solved over Connes-embeddable groups?
\end{question}

\begin{remark}
Note, however, that a positive answer to \autoref{q1} cannot generalize. Indeed, the universal system of group equations in four variables $a,b,c,d$ with content $[a,b],[a,c],[a,d],[b,c],[b,d],[c,d]$ is not solvable because $[a,c]=[a,d]=[b,c]=[b,d]=\grpone$ implies that $[a,b]$ commutes with $[c,d]$.
\end{remark}
Classically, criteria for solvability of systems of equations over groups have often been given in terms of the associated $2$-dimensional cell complex and not the presentation $2$\nobreakdash-complex or its homotopy type. In this setting, the coefficients from $G$ are allowed to be arbitrary in the equations but cancellation among the variables is controlled by the associated 2\nobreakdash-complex. If all systems of equations associated to a given cell complex are solvable, the complex is called \emph{Kervaire}.
Of course, the associated cell complex must be Kervaire if the associated universal system of group equations is solvable. But the opposite implication is not always true, as the following example shows.

\begin{example} \label{gersten} Gersten gave this example in \cite{MR888882}.
Consider systems of equations in variables $a,b,c,d,t$ with content $a^2, b^2, c^2,d^2, abtcdt^{-1}$. The presentation complex of the associated group presentation is Kervaire, as can be shown with the methods of Brick \cite{brick}. However, the system of equations
\[a^2g_1,\quad b^2g_1,\quad c^2t^{-1}g_1 t,\quad d^2t^{-1}g_1 t,\quad abtcdt^{-1}g_2\]
with constants $g_1,g_2 \in G$
is in general not solvable over $G$. Indeed, from the first four equations it follows that $a,b,tct^{-1},tdt^{-1}$ must all commute with $g_1$. Thus, if $[g_1,g_2]\neq \grpone$, there is no solution. In particular, the associated cell complex, which has some extra tweaks due to the appearance of $t^{-1}t$ in the attaching maps of two of its faces, is not Kervaire, even though it is homotopy equivalent to the presentation complex from above.

Gersten also noted that the above cell complexes are Cockcroft, i.e.\ the second Hurewicz map $\pi_2(X)\to H_2(X)$ is zero. This homotopy invariant condition is therefore insufficient for solvability of universal systems of group equations. Compare this to the slightly stronger condition appearing in \autoref{thm:main-result}.
\end{example}

The previous example shows that some information is lost by grouping systems of equations into classes according to their content instead of the associated $2$-complex.
But it also shows that the property of being a Kervaire complex is not invariant under homotopy equivalence, while our coarser notion turns out to behave better in this regard. Indeed, the sufficient condition in \autoref{thm:main-result} depends only on the homotopy type of the presentation complex, and there is some hope that the solvability of a universal system of group equations also depends only on the (simple?) homotopy type of the presentation complex.
In fact, it is easy to see that the following operations on group presentations have no effect on whether the corresponding universal system of group equations is solvable:

\begin{itemize}
\item Nielsen transformations on the relations,
\item Nielsen transformations on the generators,
\item introduction of a new generator and a new relation in which the new generator occurs exactly once, and
\item conjugation of a relation with a word in the generators
\end{itemize}

On the level of $2$-complexes, the equivalence relation generated by these operations is simple homotopy equivalence with the restriction that all elementary expansions and contractions must be of dimension $\leq 3$. The question whether this is the same as simple homotopy equivalence is closely related to the Andrews--Curtis Conjecture, see \cite{MR813099} for background.
We also recall that any two presentations of the same group are related by Tietze transformations. Hence, they can be related by the above operations after adding to each presentation a sufficient number of trivial relations.

\begin{question}
To what extent does the solvability of a universal system of equations with variables $x_1,\dots,x_n$ and content $\varepsilon(w_1),\dots,\varepsilon(w_k)$ depend only on the associated group $\Gamma=\langle x_1,\dots,x_n\mid\varepsilon(w_1),\dots,\varepsilon(w_k)\rangle$ and on the deficiency of the presentation?
\end{question}

Another line of questions concerns the nature of the group in which a system of equations is solved. In contrast to the approach of \cite{klyachkothom2017new}, our method is not able to address the second part of \autoref{conj:generalized-KLC}: Whether, if the group $G$ is finite, the group $H$ can also be chosen to be finite. For a single equation this part of the conjecture is still open. However, for systems of equations we have the following counterexample.

\begin{example}\label{exm:infinite-overgroup}
Let $G$ be a non-trivial finite group, $g \in G$ non-trivial, and consider the following system of equations with three variables $a,b,c$:
\[(bab^{-1})a (bab^{-1})^{-1}=a^2,\quad [a,c] = g.\]
The first equation is Baumslag's example that forces the element $a$ to either be trivial or have infinite order. Since the former is prohibited by the second equation, there can be no solution in a finite group.
However, the presentation complex of $\langle a,b,c \mid (bab^{-1})a (bab^{-1})^{-1}a^{-2}, [a,c] \rangle$ is aspherical by \cite{MR627092}*{Theorem 3.1} and hence our theorem applies.
\end{example}

Note that the group $\Gamma=\langle a,b,c \mid (bab^{-1})a (bab^{-1})^{-1}a^{-2}, [a,c] \rangle$ of the previous example is not residually finite. We suspect that the second part of \autoref{conj:generalized-KLC} still holds for systems of equations if the presented group is residually finite. In general, however, the semidirect product constructed in our proof is not even known to be Connes-embeddable itself. This leaves a lot of open questions.

\begin{question}
Under what conditions can a system of equations with constants in a Connes-embeddable group $G$ be solved in a Connes-embeddable group containing $G$? 
\end{question}
\begin{question}
Under what conditions can a system of equations with constants in a finite group $G$ be solved in a finite group containing $G$?
\end{question}

\section*{Acknowledgments}

We thank the anonymous referee for careful remarks that led to a structural improvement of the proof of the main result and new directions as discussed in \autoref{newrem}. This research was supported by the ERC Consolidator Grant No.\ 681207.

\end{document}